\documentclass[11pt,english]{amsart}

\usepackage[T1]{fontenc}
\usepackage{float}

\usepackage[english]{babel}
\usepackage{amssymb,url,xspace}
\usepackage{graphicx}
\usepackage{subfig}
\usepackage{xcolor}
\usepackage{textcmds}
\usepackage[section]{placeins}
\usepackage{hyperref}
\usepackage[skip=0.333\baselineskip]{caption}
\usepackage{subcaption}

\usepackage{amsmath}
\usepackage{amssymb}
\usepackage{amsfonts}
\usepackage{amsthm}
\usepackage[inkscapelatex=false]{svg}
\usepackage{calc}
\usepackage{xcolor}
\usepackage{epigraph}
\usepackage{dsfont}
\usepackage{bm}
\usepackage{enumitem}
\usepackage{float}

\theoremstyle{plain}
\newtheorem{theo}{Theorem}

\newtheorem{prop}[theo]{Proposition}

\theoremstyle{definition}
\newtheorem{defi}[theo]{Definition}

\theoremstyle{remark}
\newtheorem{rema}[theo]{Remark}
\newtheorem{exem}{Example}

\numberwithin{equation}{section}
\numberwithin{theo}{section}

\newcommand{\interior}[1]{%
  {\kern0pt#1}^{\mathrm{o}}%
}

\begin{document}

\title[]{Chainmail links, Dehn surgery number and $10/8$}
\date{\today}

\author{Soheil Azarpendar}
\address{Mathematical Institute, University of Oxford, Andrew Wiles Building,
		Radcliffe Observatory Quarter, Woodstock Road, Oxford, OX2 6GG, UK}
\email{azarpendar@maths.ox.ac.uk}

\begin{abstract}
   Liu and Piccirillo \cite{liu2024bounding} developed a combinatorial argument that employs Furuta’s 10/8-theorem to demonstrate that certain manifolds cannot be obtained via Dehn surgery on a knot. We extend their approach by creating additional examples using chainmail links.
\end{abstract}

\maketitle
\tableofcontents

\section{Introduction}\label{Section:Intro}
The celebrated Lickorish-Wallace theorem asserts that every closed, oriented 3-manifold can be obtained by performing Dehn surgery on a link in the three-sphere. The \emph{Dehn surgery number} of a 3-manifold $M$ is defined as the minimal number of components of a link in $S^3$ that admits a Dehn surgery homeomorphic to $M$. As with other invariants defined as a minimum, the Dehn surgery number is generally difficult to compute.\\ 

Classical lower bounds for the Dehn surgery number are derived from the fundamental group and homology (see Remark \ref{Remark:weight}). Accordingly, we focus on 3-manifolds that are irreducible, closed, oriented, and have the integral homology of some lens space $L(r, 1)$ for $r \in \mathbb{Z}$. In the rest of this note, we will assume all 3-manifolds have the aforementioned properties.\\

Various techniques have been employed to identify such 3-manifolds with Dehn surgery number of at least two. These include gauge theory \cite{LinesBoyer+1990+181+220,auckly1997surgery}, Heegaard Floer homology \cite{hom2016surgery,hom2018note,McCoy2014NonintegerSA}, and the $SU(2)$ character variety of the fundamental group \cite{sivek2022surgery}. To date there are no examples in the literature of such 3-manifolds with surgery number
greater than two. Daemi and Miller-Eismeier announced examples in 2022, and their preprint is forthcoming.\\

Recently, Liu and Piccirillo \cite{liu2024bounding} proposed a new method for producing examples of 3-manifolds with Dehn surgery numbers greater than one. Their method uses Kaplan's characteristic sublink calculus and Furuta's 10/8-theorem. They provide one specific family of examples for their method, presented in Theorem \ref{Theorem:LiuPiccirilomain}.\\ 

\begin{theo}\label{Theorem:LiuPiccirilomain}\cite{liu2024bounding}
Let $L_n$ be the 2-component link shown in Figure \ref{LiuPicciriolLink}. For any odd integers $p, q$ and odd integer $n$ sufficiently large, the surgery manifold $S^3_{p,q}(L_n)$ has Dehn surgery number two.
\end{theo}

\begin{figure}[h]
\def\svgscale{0.4}
\centerline{
\begingroup%
  \makeatletter%
  \providecommand\color[2][]{%
    \errmessage{(Inkscape) Color is used for the text in Inkscape, but the package 'color.sty' is not loaded}%
    \renewcommand\color[2][]{}%
  }%
  \providecommand\transparent[1]{%
    \errmessage{(Inkscape) Transparency is used (non-zero) for the text in Inkscape, but the package 'transparent.sty' is not loaded}%
    \renewcommand\transparent[1]{}%
  }%
  \providecommand\rotatebox[2]{#2}%
  \newcommand*\fsize{\dimexpr\f@size pt\relax}%
  \newcommand*\lineheight[1]{\fontsize{\fsize}{#1\fsize}\selectfont}%
  \ifx\svgwidth\undefined%
    \setlength{\unitlength}{372.28089821bp}%
    \ifx\svgscale\undefined%
      \relax%
    \else%
      \setlength{\unitlength}{\unitlength * \real{\svgscale}}%
    \fi%
  \else%
    \setlength{\unitlength}{\svgwidth}%
  \fi%
  \global\let\svgwidth\undefined%
  \global\let\svgscale\undefined%
  \makeatother%
  \begin{picture}(1,1.14998823)%
    \lineheight{1}%
    \setlength\tabcolsep{0pt}%
    \put(0,0){\includegraphics[width=\unitlength,page=1]{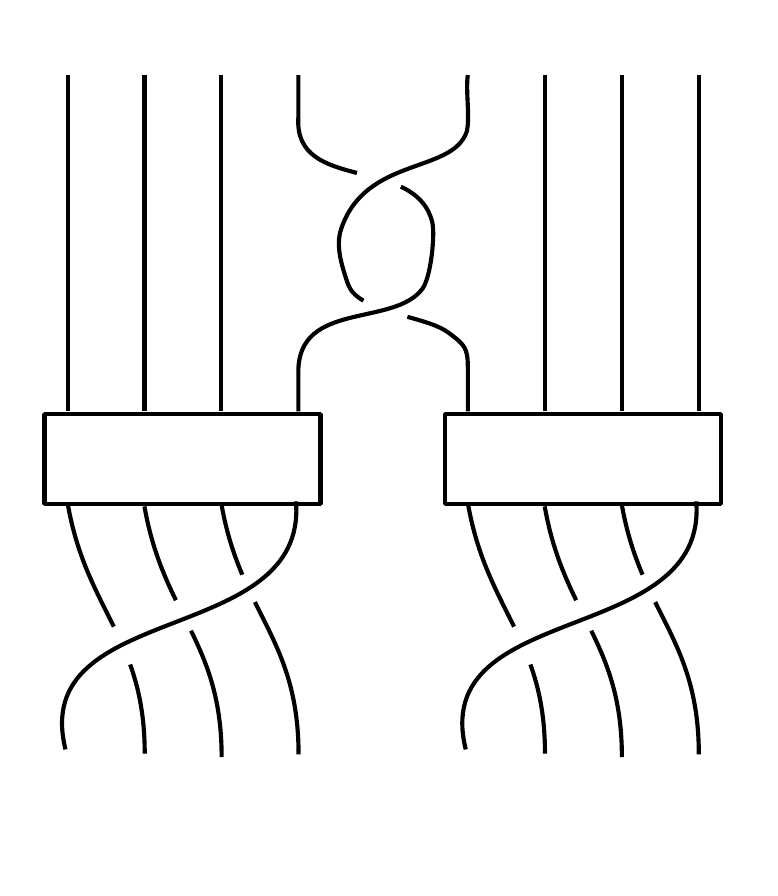}}%
    \put(0.22623748,0.5422397){\color[rgb]{0,0,0}\makebox(0,0)[lt]{\lineheight{1.25}\smash{\begin{tabular}[t]{l}$1$\end{tabular}}}}%
    \put(0.74928946,0.53823256){\color[rgb]{0,0,0}\makebox(0,0)[lt]{\lineheight{1.25}\smash{\begin{tabular}[t]{l}$1$\end{tabular}}}}%
    \put(0,0){\includegraphics[width=\unitlength,page=2]{TheirLink.pdf}}%
    \put(0.22080001,0.0463945){\color[rgb]{0,0,0}\makebox(0,0)[lt]{\lineheight{1.25}\smash{\begin{tabular}[t]{l}$n$\end{tabular}}}}%
    \put(0.72849946,0.04984403){\color[rgb]{0,0,0}\makebox(0,0)[lt]{\lineheight{1.25}\smash{\begin{tabular}[t]{l}$n$\end{tabular}}}}%
  \end{picture}%
\endgroup%
}
\caption{The link $L_n$ is obtained by taking the braid closure of the braid depicted in this figure.}\label{LiuPicciriolLink}
\end{figure}

It is possible that extensions of Liu and Piccirillo’s methods could produce 3-manifolds with Dehn surgery number greater than two. To explore this direction, it is valuable to develop additional examples. In this note, we combine Liu and Piccirillo’s approach with the author's previous work \cite{azarpendar2023negative} on chainmail links, introducing a systematic method for constructing such examples. These examples arise as surgeries on chainmail links. An application of this construction appears in Theorem \ref{Theorem:Mainexample}.

\begin{theo}\label{Theorem:Mainexample}
Let $K_n$ be the knot shown in Figure \ref{Figure:Myexample}. Then for $n$ sufficiently large, the branched double cover $\Sigma(S^3,K_n)$ is not a surgery on a knot in $S^3$.
\end{theo}

This note is organized as follows. Section \ref{Section:Chainmaillinks} defines chainmail links and describes the Tait surgery diagrams of branched double covers. Section \ref{Section:Kaplan} reviews Kaplan’s algorithm and explains how it can be simplified for chainmail links. Section \ref{Section:Strategy} outlines a general strategy, inspired by Liu and Piccirillo’s framework, for producing 3-manifolds with Dehn surgery number greater than one. Finally, Section \ref{Section:Examples} constructs explicit examples that illustrate the method.

\begin{figure}[h]
\centering
\begin{center}
\includegraphics[scale=0.4]{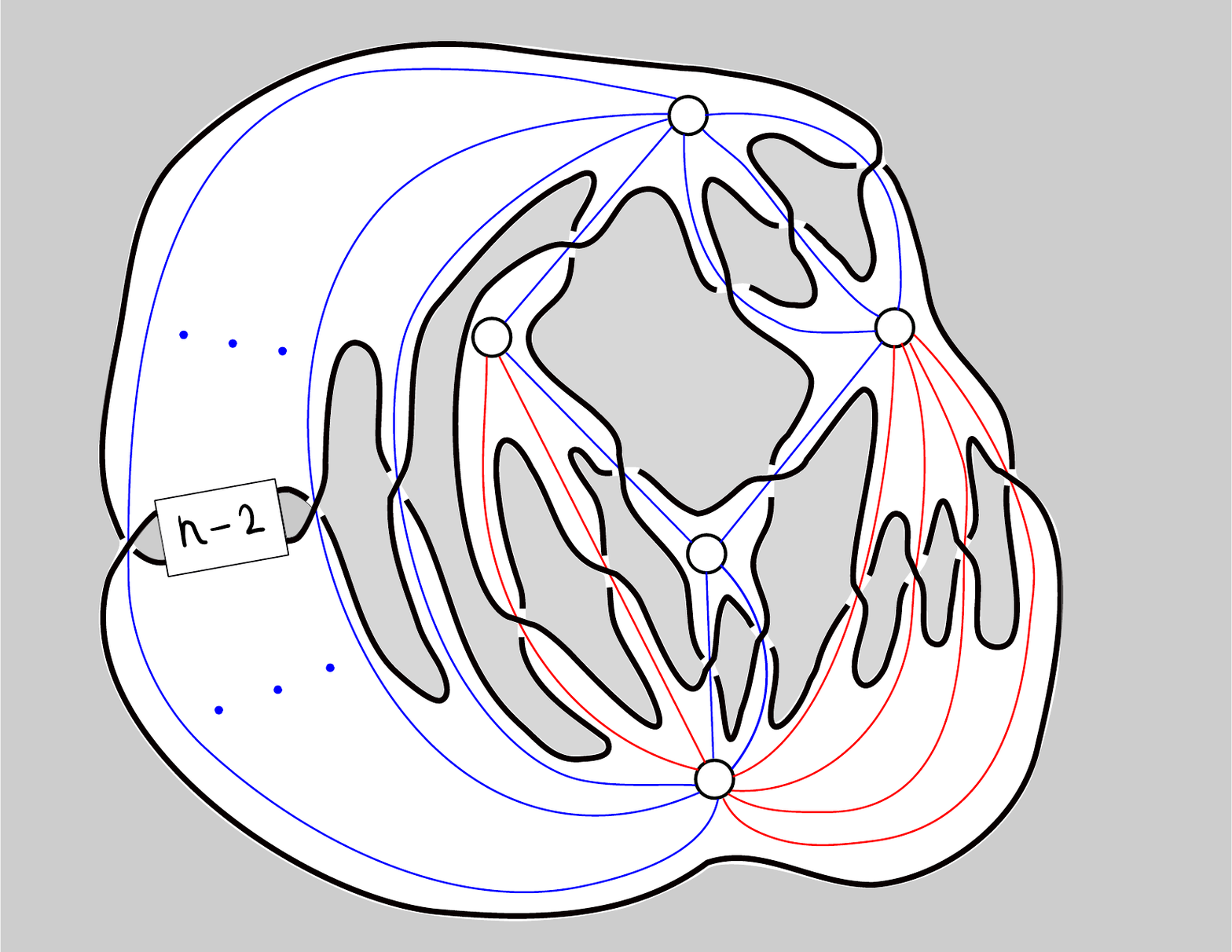}
\end{center}
\caption{Diagram of the knot $K_n$. The knot contains a twist region with $2n-1$ crossings. The checkerboard coloring and the white Tait graph can also be seen in the picture. One can see that the Tait graph is same as $W_n$ in Figure \ref{fig:Wn}.}\label{Figure:Myexample}
\end{figure}

\section{Acknowledgement}
I am deeply grateful to my advisor, Professor András Juhász, whose patience and guidance made this project possible. I also thank Professor Lisa Piccirillo and Professor Beibei Liu for their insightful feedback and encouragement

\section{Chainmail links}\label{Section:Chainmaillinks}
We begin by defining the essential terminology. Let $D$ be a weighted and signed plane multigraph. We denote the sets of vertices and edges of $D$ by $V_D$ and $E_D$, respectively, where: 
$$V_D=\{v_1,\cdots,v_n\}, \ \text{and}, $$
$$E_D=\{e_1,\cdots,e_m\}.$$
Each vertex $v_i \in V_D$ is assigned an integer weight $w_{D}(v_i) \in \mathbf{Z}$, and edge $e_k \in E_D$ is assigned a sign $\mu_{D}(e_k) \in \{+ , -\}$. \\

\begin{defi}
 Let $D$ be a weighted and signed plane multigraph. We construct the framed link $L_D$ as follows: 
\begin{itemize}
    \item For each vertex $v_i \in V_D$, include a counterclockwise-oriented unknot component $L_i$ centered at $v_i$, with framing $w_{D}(v_i)$.
    \item For each edge $e_k \in E_D$ connecting $v_i$ and $v_j$, add a clasp between $L_i$ and $L_j$. The clasp is positive if $\mu_{D}(e_k) = +$, and negative if $\mu_{D}(e_k) = -$.
\end{itemize}
We refer to $L_D$ as the \emph{chainmail link associated to $D$}.
\end{defi}

An example of a weighted and signed plane multigraph and its corresponding chainmail link is depicted in Figures \ref{chainmail1} and \ref{chainmail2}.\\ 

Let $V' \subset V_D$ be a subset of the vertices of $D$, and let $D'=D[V']$ denote the subgraph induced by $V'$. The chainmail link associated to $D'$ forms a sublink of $L_D$, given explicitly by
$$L_{D'} = \bigcup_{v_i \in V'} L_i.$$
In this way, every sublink  of $L_D$ arises from a subgraph of $D$.\\ 

\begin{figure}[h]
\def\svgscale{0.4}
\centerline{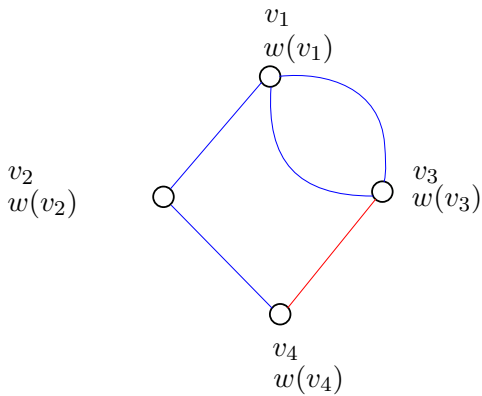}
\caption{A weighted and signed plane multigraph. The edges with color blue (resp.~red) has sign $+$ (resp.~$-$).}\label{chainmail1}
\end{figure}

Since $L_D$ is a framed link, it represents a surgery diagram of a $3-$manifold, denoted $Y_D$, and a Kirby diagram of a simply connected filling of $Y_D$, denoted $X_D$. We refer to any surgery diagram (resp.~Kirby diagram) whose underlying link is a chainmail link as a \emph{chainmail surgery diagram (resp.~chainmail Kirby diagram)}.\\

Polyak \cite{Polyak} showed that all $3-$manifolds admit a chainmail surgery diagram. A notable instance is the \emph{Tait surgery diagram} of a branched double cover of a link, first constructed by Ozsváth and Szabó \cite{SObranch} (see also \cite{azarpendar2023negative}, \cite{owens2022knots} and \cite{Greene2008AST}). We briefly review this construction below.\\ 

\begin{figure}[h]
\def\svgscale{0.4}
\centerline{
\begingroup%
  \makeatletter%
  \providecommand\color[2][]{%
    \errmessage{(Inkscape) Color is used for the text in Inkscape, but the package 'color.sty' is not loaded}%
    \renewcommand\color[2][]{}%
  }%
  \providecommand\transparent[1]{%
    \errmessage{(Inkscape) Transparency is used (non-zero) for the text in Inkscape, but the package 'transparent.sty' is not loaded}%
    \renewcommand\transparent[1]{}%
  }%
  \providecommand\rotatebox[2]{#2}%
  \newcommand*\fsize{\dimexpr\f@size pt\relax}%
  \newcommand*\lineheight[1]{\fontsize{\fsize}{#1\fsize}\selectfont}%
  \ifx\svgwidth\undefined%
    \setlength{\unitlength}{538.54002945bp}%
    \ifx\svgscale\undefined%
      \relax%
    \else%
      \setlength{\unitlength}{\unitlength * \real{\svgscale}}%
    \fi%
  \else%
    \setlength{\unitlength}{\svgwidth}%
  \fi%
  \global\let\svgwidth\undefined%
  \global\let\svgscale\undefined%
  \makeatother%
  \begin{picture}(1,0.78084853)%
    \lineheight{1}%
    \setlength\tabcolsep{0pt}%
    \put(0,0){\includegraphics[width=\unitlength,page=1]{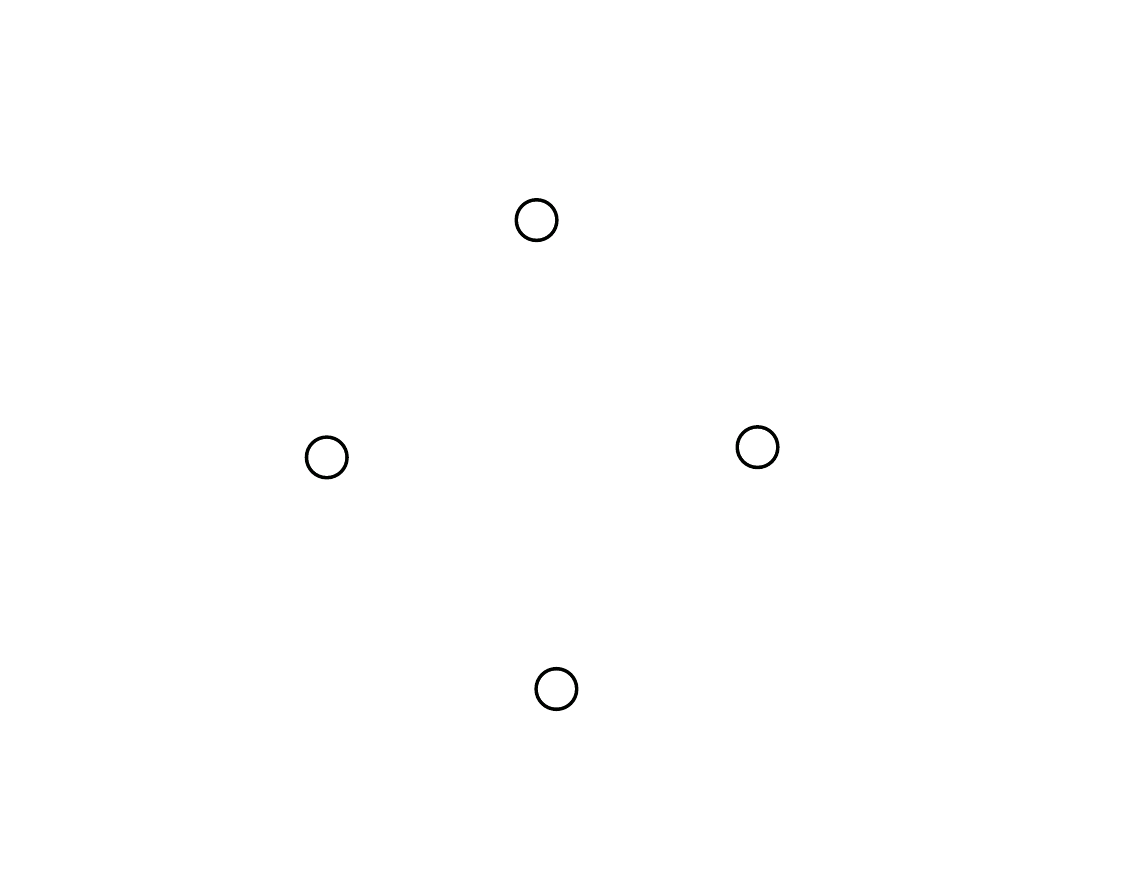}}%
    \put(0.4551219,0.04121207){\color[rgb]{0,0,0}\makebox(0,0)[lt]{\lineheight{1.25}\smash{\begin{tabular}[t]{l}$w(v_4)$\end{tabular}}}}%
    \put(0.03702094,0.37210616){\color[rgb]{0,0,0}\makebox(0,0)[lt]{\lineheight{1.25}\smash{\begin{tabular}[t]{l}$w(v_2)$\end{tabular}}}}%
    \put(0.46558287,0.74055285){\color[rgb]{0,0,0}\makebox(0,0)[lt]{\lineheight{1.25}\smash{\begin{tabular}[t]{l}$w(v_1)$\end{tabular}}}}%
    \put(0.8948935,0.35788513){\color[rgb]{0,0,0}\makebox(0,0)[lt]{\lineheight{1.25}\smash{\begin{tabular}[t]{l}$w(v_3)$\end{tabular}}}}%
    \put(0,0){\includegraphics[width=\unitlength,page=2]{PicciriloLiuChainMail.pdf}}%
  \end{picture}%
\endgroup%
}
\caption{The associated chainmail link (in green)}\label{chainmail2}
\end{figure}

Choose a diagram of the link $L$ and a checkerboard coloring of its regions. There is a standard way to assign signs to the crossings in a colored diagram which is described in Figure~\ref{sign}.\\

The \emph{white Tait graph} $W$ is defined as a weighted and signed plane multigraph where: 
\begin{itemize}
    \item Each white region in the diagram corresponds to a vertex.
    \item An edge $e_c$ is drawn between two regions for any common crossing $c$ on their boundary.
    \item The sign of the edge $e_c$ is $-\mu(c)$.
    \item The weight of a vertex $v$ is defined as 
    $$w_{W}(v) := -\sum_{\substack{e \in E_W \\ e \ \text{adjacent to} \ v}} \mu_{W}(e).$$
\end{itemize}

\begin{figure}[h]
\def\svgscale{0.6}
\centerline{
\begingroup%
  \makeatletter%
  \providecommand\color[2][]{%
    \errmessage{(Inkscape) Color is used for the text in Inkscape, but the package 'color.sty' is not loaded}%
    \renewcommand\color[2][]{}%
  }%
  \providecommand\transparent[1]{%
    \errmessage{(Inkscape) Transparency is used (non-zero) for the text in Inkscape, but the package 'transparent.sty' is not loaded}%
    \renewcommand\transparent[1]{}%
  }%
  \providecommand\rotatebox[2]{#2}%
  \newcommand*\fsize{\dimexpr\f@size pt\relax}%
  \newcommand*\lineheight[1]{\fontsize{\fsize}{#1\fsize}\selectfont}%
  \ifx\svgwidth\undefined%
    \setlength{\unitlength}{398.1331989bp}%
    \ifx\svgscale\undefined%
      \relax%
    \else%
      \setlength{\unitlength}{\unitlength * \real{\svgscale}}%
    \fi%
  \else%
    \setlength{\unitlength}{\svgwidth}%
  \fi%
  \global\let\svgwidth\undefined%
  \global\let\svgscale\undefined%
  \makeatother%
  \begin{picture}(1,0.37352217)%
    \lineheight{1}%
    \setlength\tabcolsep{0pt}%
    \put(0,0){\includegraphics[width=\unitlength,page=1]{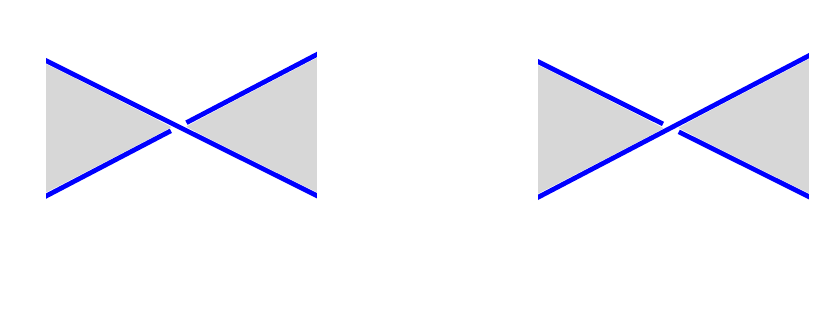}}%
    \put(0.12319831,0.03894915){\color[rgb]{0,0,0}\makebox(0,0)[lt]{\lineheight{1.25}\smash{\begin{tabular}[t]{l}$\mu=+1$\end{tabular}}}}%
    \put(0.72317562,0.03895458){\color[rgb]{0,0,0}\makebox(0,0)[lt]{\lineheight{1.25}\smash{\begin{tabular}[t]{l}$\mu=-1$\end{tabular}}}}%
  \end{picture}%
\endgroup%
}
\caption{The sign of a crossing in a colored diagram}\label{sign}
\end{figure}

Choosing an arbitrary root vertex $r \in V_W$ and removing it from $W$ yields the \emph{reducred white Tait graph $\widetilde{W}$}. The Tait surgery diagram $L_{\widetilde{W}}$ is the chainmail link associated to $\widetilde{W}$.
\begin{theo}[\cite{SObranch}]\label{Theorem:OS}
For any link $L$, and any reduced white Tait graph of $L$ such as $\widetilde{W}$, we have
$$\Sigma(S^3,L)\cong Y_{\widetilde{W}}.$$
That is, the chainmail link associated to $\widetilde{W}$ is a surgery diagram for the branched double cover of $S^3$ over $L$. We refer to this as the Tait surgery diagram.
\end{theo}

\section{Kaplan algorithm on chainmail Kirby diagrams}\label{Section:Kaplan}
Now we recall a computation of the author \cite{azarpendar2023negative} concerning the outcome of Kaplan’s algorithm when applied to chainmail Kirby diagrams. Throughout this section, we continue to use the notation established in Section \ref{Section:Chainmaillinks}.\\

Kaplan showed that the spin structures on $Y_D$ are in one-to-one correspondence with characteristic sublinks of $L_D$. Given a spin structure $\mathfrak{s} \in \text{Spin}(Y_D)$, let $L_{\mathfrak{s}} \subset L_{D}$ be the corresponding characteristic sublink. As noted in Section \ref{Section:Chainmaillinks}, sublinks of $L_D$ correspond to induced subgraphs of $D$. Let $S$ denote the subgraph associated to $L_{\mathfrak{s}}$ i.e.
$$ V_S \subseteq V_D\ , \ S=D[V_S] \ \text{and} \ L_{S} = L_{\mathfrak{s}}.$$
We call such subgraphs \emph{characteristic subgraphs} of $D$.\\

Kaplan introduced an algorithm to turn $L_S$ to the empty sublink through Kirby moves, producing a simply-connected spin filling $X_{D,\mathfrak{s}}$ of $(Y_D,\mathfrak{s})$.\\

The author \cite{azarpendar2023negative} simplified the Kaplan's algorithm for chainmail Kirby diagrams, enabling computation of the second Betti number and signature of $X_{D,\mathfrak{s}}$. We recall this result in Proposition \ref{Proposition:compute}, and give a sketch of the proof. For more details we refer the reader to \cite{azarpendar2023negative}.\\

Before stating the result, we need to introduce a piece of notation in Definition \ref{Definition:edgesum}.

\begin{defi}\label{Definition:edgesum}
Let $D$ be a weighted and signed multigraph. For any pair of vertices $v,v' \in V_{D}$, we define $\mu_{D}(E(v,v'))$ to be the signed count of edges between $v$ and $v'$.
\end{defi}

\begin{prop}\label{Proposition:compute}
Let $L_{S}$ and $S$ respectively be the characteristic sublink and subgraph corresponding to $\mathfrak{s} \in \text{Spin}(Y_D)$. We define $f_{\mathfrak{s}}$ as follows: 
$$f_{\mathfrak{s}}:=\sum_{v \in V(S)} w_{D}(v) + 2 \sum_{v,v' \in V(S)} \mu_{D}(E(v,v').$$
Then there exist a spin filling $X_{D,\mathfrak{s}}$ of $(Y_D, \mathfrak{s})$ such that
$$b_2(X_{D,\mathfrak{s}})= b_2(X_D) + |f_{\mathfrak{s}}|-2 \ \text{and} \ $$
$$\sigma(X_{D,\mathfrak{s}})= \sigma(X_D) - f_{\mathfrak{s}}.$$
\end{prop}
\begin{proof}
  
Note that, as the notation suggests, the characteristic sublink $L_S$ is the chainmail link associated to $S$.\\

We begin with a modified version of the first part of the Kaplan's algorithm (described in \cite[Definition 4.4]{azarpendar2023negative}). At each step, we pick two components $L_i,L_j$ (corresponding to vertices $v_i , v_j \in V(S)$) and we slide $L_i$ over $L_j$. The component $L_j$ is then deleted from the characteristic sublink. This process is repeated until only a single component remains.\\

As shown in \cite[Lemma 4.5]{azarpendar2023negative}, the handle slides can be performed in an order that ensures the characteristic sublink remains a chainmail link at each step. The local model of the handle slides can be seen in Figure \ref{Figure:MK1pic}. On the level of the underlying graph, the handle slide corresponds to contracting $v_i$ and $v_j$ into a single vertex. 

\begin{figure}[h]
\def\svgscale{0.6}
\centerline{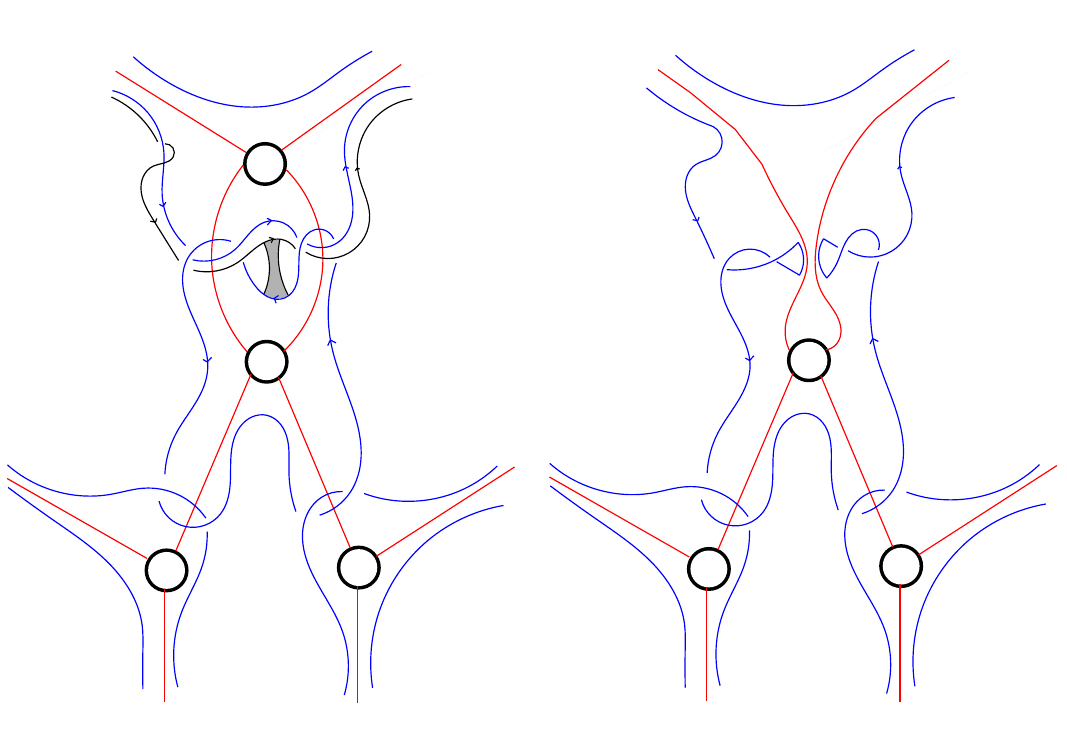}
\caption{Local model of a handle slide in a chainmail link \cite{azarpendar2023negative}}\label{Figure:MK1pic}
\end{figure}

After completing this modified version of Kaplan's algorithm the characteristic sublink becomes an unknot.\\

According to the framing change formula sliding $L_i$ over $L_j$ updates the framing on $L_i$ to 
$$w_{D}(v_i) + w_{D}(v_j) + 2 \cdot \text{lk}(L_i,L_j) = w_{D}(v_i) + w_{D}(v_j) + 2 \sum_{e \in E(v_i,v_j)} \mu_{D}(e).$$
By iterating this formula we obtain that the framing of the final unknot is
$$\sum_{v \in V(S)} w_{D}(v) + 2 \sum_{v,v' \in V(S)} \mu_{D}(E(v,v')) = f_{\mathfrak{s}}.$$

Up to this point, we have only applied handle slides, so we still have a Kirby diagram of $X_{D}$.\\

Now to construct $X_{D,\mathcal{s}}$, we perform a sequence of blow-ups and a blow-down. Blowing up $|f_{\mathfrak{s}}|-1$ meridians converts the framing of the characteristic unknot to $\pm1$, and blowing down this unknot then yields $X_{D,\mathfrak{s}}$.\\ 

Computation of the second Betti number and signature of $X_{D,\mathfrak{s}}$ is straightforward using the description above.   
\end{proof}

\section{Strategy for bounding the Dehn surgery number}\label{Section:Strategy}
In this section, we introduce a method for building examples of 3-manifolds with Dehn sugery number higher than one, using Liu and Piccirillo's argument. This is explained in Proposition \ref{Proposition:strategy}.

\begin{prop}\label{Proposition:strategy} Let $\{D_n\}$ be a 
 family of weighted and signed plane multigraphs for $n \in \mathbf{N}$, such that: 
 \begin{enumerate}[label=(\roman*)]
     \item $\forall n \in \mathbf{N}: |H_1(Y_{D_n})| = 0 \ \text{(mod 2)}$
     \item $ \exists M$ such that $\forall n \in \mathbf{N}: |H_1(Y_{D_n})|<M,$
     \item $ \exists B$ such that $\forall n \in \mathbf{N}: |V_{D_n}|<B,$
     \item For any sequnce $(\mathfrak{s}_{n})_{n \in \mathbf{N}}$ with $\mathfrak{s}_{n} \in \text{Spin}(Y_{D_n})$, we have: 
     $$\lim_{n \to \infty} |f_{\mathfrak{s}_{n}}| = +\infty.$$
 \end{enumerate}
 Then for sufficiently large $n$, the 3-manifold $Y_{D_n}$ can not be obtained by Dehn surgery on a knot.
\end{prop}
\begin{proof}
We just go through Liu and Piccirillo's argument. Assume the contrary, that there exist knots $K_n$ and integers $s_n , t_n \in \mathbf{Z}$ such that
\begin{equation}\label{Equation:Dehnsurgery}
Y_{D_n} = S^3_{s_n/t_n}(K_n).
\end{equation}
 Note that we have 
 $$|H_{1}(Y_{D_n})| = |H_{1}(S^3_{s_n/t_n}(K_n))|=s_{n}.$$
Hence $s_n$ must be even, and 
 $$\forall n \in \mathbf{N}: |s_n| < M.$$ 
 We can use a trick to turn the rational surgery in Equation \ref{Equation:Dehnsurgery} to an integral one. Based on the solved cases of the cabling conjecture, we can write:
 $$Y_{D_n} \# L(t_n,s_n) = S^3_{s_n/t_n}(K_n) \# L(t_n,s_n) = S^3_{s_nt_n}(T_{t_n,s_n}(K_n)),$$
 where $T_{t_n,s_n}(K_n)$ is the $(t_n,s_n)-$cable of the knot $K_n$. To simplify the notations we use $\widetilde{K}_n$ to denote the cable $T_{t_n,s_n}(K_n)$ in the rest of the argument.\\

Now, the trace $X_{s_nt_n}(\widetilde{K}_n)$ is a spin filling of $Y_{D_n} \# L(t_n,s_n)$ for a spin structure $\mathfrak{s}'_n$. The spin structure $\mathfrak{s}'_n$ comes from taking the connected sum of a spin structure $\mathfrak{s}_{n} \in \text{Spin}(Y_{D_n})$, with the unique spin structure of $L(t_n,s_n)$.\\

Now we can construct another spin filling of $(Y_{D_n} \# L(t_n,s_n),\mathfrak{s}'_n)$ as follows:
$$X'_{\mathfrak{s}_{n}} := X_{D_{n},\mathfrak{s}_{n}} \ \# \ V_{t_n,s_n}.$$
Here, $X_{D_{n},\mathfrak{s}_{n}} $ is the result of Kaplan algorithm described in Section \ref{Section:Kaplan}, and the 4-manifold $V_{t_n,s_n}$ is a spin filling of $L(t_n,s_n)$ with the property that 
$$b_2(V_{t_n,s_n}) \leq |s_n| \ \text{and} \ |\sigma(V_{t_n,s_n})| \leq |s_n|.$$

Finally we can consider the closed spin 4-maifold 
$$\widetilde{X}_{n} := X'_{\mathfrak{s}_{n}}  \ \cup \ - X_{s_nt_n}(\widetilde{K}_n).$$
Now, using the Mayer–Vietoris sequence and Novikov additivity, we compute the signature and second Betti number of $\widetilde{X}_{n}$ as follows:
$$\sigma(\widetilde{X}_{n}) = \sigma(X_{D_n,\mathfrak{s}_{n}})+\sigma(V_{t_n,s_n})-\sigma(X_{s_nt_n}(\widetilde{K}_n)), \ \text{and},$$
$$b_{2}(\widetilde{X}_{n}) = b_2(X_{D_n,\mathfrak{s}_{n}}) + b_2(V_{t_n,s_n}) + b_2(X_{s_nt_n}(\widetilde{K}_n)).$$
Based on what we discussed up to here we will have :
$$ \lim_{n \to \infty} \frac{|\sigma(\widetilde{X}_{n})|}{b_{2}(\widetilde{X}_{n})} = \lim_{n \to \infty} \frac{|\sigma(X_{D_n,\mathfrak{s}_n})|}{b_2(X_{D_n,\mathfrak{s}_n})}$$
And now we can use Proposition \ref{Proposition:compute} to conclude
\begin{equation}\label{Equation:finallimit}
\lim_{n \to \infty} \frac{|\sigma(X_{D_n,\mathfrak{s}_n})|}{b_2(X_{D_n,\mathfrak{s}_n})} = \lim_{n \to \infty} \frac{|\sigma(X_{D_n})- f_{s_{n}}|}{b_2(X_{D_n}) + |f_{\mathfrak{s}_n}|-2} =\lim_{n \to \infty} \frac{|f_{\mathfrak{s}_n}|}{|f_{\mathfrak{s}_n}|} = 1.
\end{equation}
The final conclusion follows from condition $(iii)$, as $|\sigma(X_{D_n})|$ and $b_2(X_{D_n})$ are bounded above by $|V_{D_n}|$.\\

For sufficently large $n$, the Equality \ref{Equation:finallimit} contradicts Furuta's 10/8 inequality. 
\end{proof}
It is not clear whether an example satisfying conditions $(i)-(iv)$ exists. We build explicit examples in Section \ref{Section:Examples}.

\section{Examples}\label{Section:Examples}
Proposition \ref{Proposition:genex1} gives examples of a family $D_n$ which satisfy conditions of Proposition \ref{Proposition:strategy}. To prepare for this, we begin by recalling the definition of the Laplacian matrix of a weighted and signed multigraph.
\begin{defi}
    Let $D$ be a weighted and signed multigraph with vertex set $V_D=\{v_1,\cdots,v_n\}$. The Laplacian matrix $A^D \in \mathbb{Z}^{n \times n}$ is defined by:
    \begin{equation}
    \begin{cases}
      A^D_{ii} = w_{D}(v_i) \ \text{for} \  1 \leq i \leq n\\
      A^D_{ij}=A^D_{ji}=\mu_{D}(E(v_i,v_j)) \ \text{for} \  1 \leq i < j \leq n
    \end{cases}       
\end{equation}
\end{defi}

\begin{prop}\label{Proposition:genex1}
    Let $D$ be a weighted and signed plane multigraph satisfying the following properties:
     \begin{enumerate}
     \item $\text{det}(A^D) \neq 0$ and $\text{det}(A^D) = 0 \ \text{(mod 2)}$ 
     \item $\exists v \in V_D $ such that $v \in V_S$ for any characteristic subgraph $S$ of $D$. 
     \item There exist two different vertices $v' , v'' \in V_D$ such that: 
     $$\forall x \in V_D \setminus \{v,v',v''\} \ : \ \mu_{D}(E(x,v')) = \mu_{D}(E(x,v'')) \ \text{and}$$
     $$\mu_{D}(E(v',v'')) = w_{D}(v')=w_{D}(v'').$$
     \end{enumerate}
     Define $D_n$ to be the same weighted and signed plane multigraph with only change being that we set 
     \begin{equation}\label{weighinc}
        w_{D_n}(v)=w_{D}(v)-2n.
     \end{equation}
Then the family $D_n$ satisfy the conditions of Proposition \ref{Proposition:strategy}.
\end{prop}

\begin{proof}
Without loss of generality assume that $v,v',v''$ are respectively $v_1,v_2,v_3$.\\

Condition $(iii)$ obviously holds. To prove conditions $(i)$ and $(ii)$, we need to compute the $H_1(Y_{D_n})$ by applying the Mayer–Vietoris theorem (on the surgery diagram $L_{D_n}$). This gives us that 
$$|H_1(Y_{D_n})| = det(A^{D_n}).$$
Now note that the only difference between $A^{D_n}$ and $A^{D}$ is in the $(1,1)-$element where:
$$A_{11}^{D_n} = w_{D_n}(v_1) = w_{D}(v_1)-2n=A_{11}^{D} - 2n.$$
It turns out that this difference doesn't affect the determinant and we have $\text{det}(A^{D_n})=\text{det}(A^{D})$. This equality comes from the fact that the second and third rows of $A^{D_n}$ can only differ in their first entry due the third property. As a result, the determinant of $(1,1)-$minor of $A^{D_n}$ vanishes and hence the $(1,1)-$element's values doesn't appear in the expansion of determinant. This gives us that 
$$\forall n \in \mathbf{N}: |H_1(Y_{D_n})| = \text{det}(A^{D}),$$
which means conditions $(i)$ and $(ii)$ will be satisfied.\\

To prove Condition $(iv)$, first note that as $n$ changes, the characteristic subgraphs of $D_n$ remain unchanged, since the parity of the weights remain unchanged.\\

Consider any choice of $(\mathfrak{s}_{n})_{n \in \mathbf{N}}$ such that $\mathfrak{s}_{n} \in \text{Spin}(Y_{D_n})$. Let $S_n$ be the characteristic subgraph of $D_n$ associated to $\mathfrak{s}_{n}$. Due to the second property and the argument of the last paragraph, we have $v_1 \in V_{S_n}$. Based on the definition of $f_{\mathfrak{s}_{n}}$ and $D_n$ we will have 
$$f_{\mathfrak{s}_{n}} = f_{\mathfrak{s}_{0}} - 2(w_{D_n}(v) - w_{D_0}(v)) = f_{\mathfrak{s}_{0}}+2n.$$
This establishes condition $(iv)$. 
\end{proof}
\begin{exem}\label{example1}
    A very simple example of Proposition \ref{Proposition:genex1} is the following choice for $D$.\\
    
\begin{figure}[h]
\centering
\def\svgscale{0.4}
\centerline{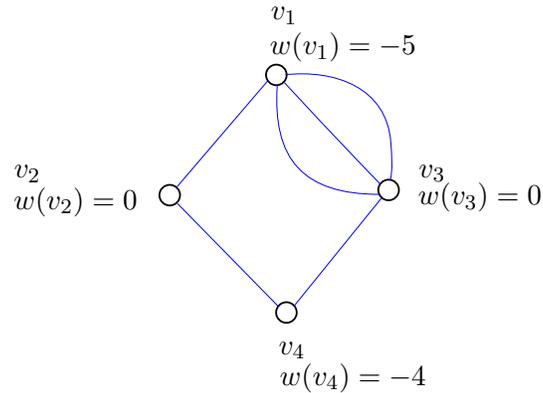}
\caption{Example of Proposition \ref{Proposition:genex1}. The weights are written in red. All of the edges has $+$ sign.}\label{fig:example1}
\end{figure}
The Laplacian matrix $A^D$ will be as follows:
$$A^D=\begin{bmatrix}
-5 & 1 & 3 & 0 \\
1 & 0 & 0 & 1 \\
3 & 0 & 0 & 1 \\
0 & 1 & 1 & -4 
\end{bmatrix}$$
As a result, $\text{det}(A^D)=4$ and hence the first assumption holds. The only characteristic subgraphs of $D$ are induced by $\{v_1,v_4\}$ and $\{v_1,v_2,v_3,v_4\}$ which means that the second assumption also holds by letting $v$ to be $v_1$. The third assumption also clearly holds by letting $v',v''$ to respectively be $v_2,v_3$.
\end{exem}
\begin{rema}\label{Remark:weight}
    As mentioned in Section \ref{Section:Intro}, the Dehn surgery number can sometimes be bounded using algebraic invariants. As a result of the Seifert-Van Kampen theorem, the weight of the fundamnetal group of a 3-manifold with a Dehn surgery number of one is equal to one.\\
    
    In general, there is no guarantee that the examples generated in Proposition \ref{Proposition:genex1} will have a fundamental group of weight one. However, some do exhibit this property as we see in Proposition \ref{Prop:weight}
    \end{rema}
    
\begin{prop}\label{Prop:weight}
    Let $\{D_{n}\}_{n \in \mathbb{N}}$ be the family of weighted, signed plane graphs constructed from the graph $D$ described in Example \ref{example1}, using the procedure in Proposition \ref{Proposition:genex1}. Then the fundamental group $\pi_1(Y_{D_n})$ has weight one for all $n \in \mathbf{N}$.
\end{prop}

\begin{proof}
        First, we show that $L_{D_n}$ are all Tait surgery diagrams. We do this by adding a root vertex and then connecting it to $V_{D_n}$ such that we end up with a graph $W_n$ with the property that
    $$w_{W_n}(v) := -\sum_{\substack{e \in E_{W_n} \\ e \ \text{adjacent to} \ v}} \mu_{W_n}(e).$$
    The graph $W_n$ can be seen in Figure \ref{fig:Wn}. Now, one can construct a knot $K_n$ by applying median construction on $W_n$, such that $W_n$ is the white Tait graph of $K_n$. This can be seen in Figure \ref{Figure:Myexample}.\\
    
    Hence $D_n$ will be the reduced white Tait graph of $K_n$, and based on Theorem \ref{Theorem:OS}, we have:  
    $$\Sigma(S^3,K_n) \simeq Y_{D_n}.$$
\begin{figure}[h]
\centering
\def\svgscale{0.4}
\centerline{
\begingroup%
  \makeatletter%
  \providecommand\color[2][]{%
    \errmessage{(Inkscape) Color is used for the text in Inkscape, but the package 'color.sty' is not loaded}%
    \renewcommand\color[2][]{}%
  }%
  \providecommand\transparent[1]{%
    \errmessage{(Inkscape) Transparency is used (non-zero) for the text in Inkscape, but the package 'transparent.sty' is not loaded}%
    \renewcommand\transparent[1]{}%
  }%
  \providecommand\rotatebox[2]{#2}%
  \newcommand*\fsize{\dimexpr\f@size pt\relax}%
  \newcommand*\lineheight[1]{\fontsize{\fsize}{#1\fsize}\selectfont}%
  \ifx\svgwidth\undefined%
    \setlength{\unitlength}{502.70037794bp}%
    \ifx\svgscale\undefined%
      \relax%
    \else%
      \setlength{\unitlength}{\unitlength * \real{\svgscale}}%
    \fi%
  \else%
    \setlength{\unitlength}{\svgwidth}%
  \fi%
  \global\let\svgwidth\undefined%
  \global\let\svgscale\undefined%
  \makeatother%
  \begin{picture}(1,0.97735164)%
    \lineheight{1}%
    \setlength\tabcolsep{0pt}%
    \put(0.11623346,0.26073899){\color[rgb]{0,0,0}\rotatebox{18.13430321}{\makebox(0,0)[lt]{\lineheight{1.25}\smash{\begin{tabular}[t]{l}$\color{blue}2n-1$\end{tabular}}}}}%
    \put(0,0){\includegraphics[width=\unitlength,page=1]{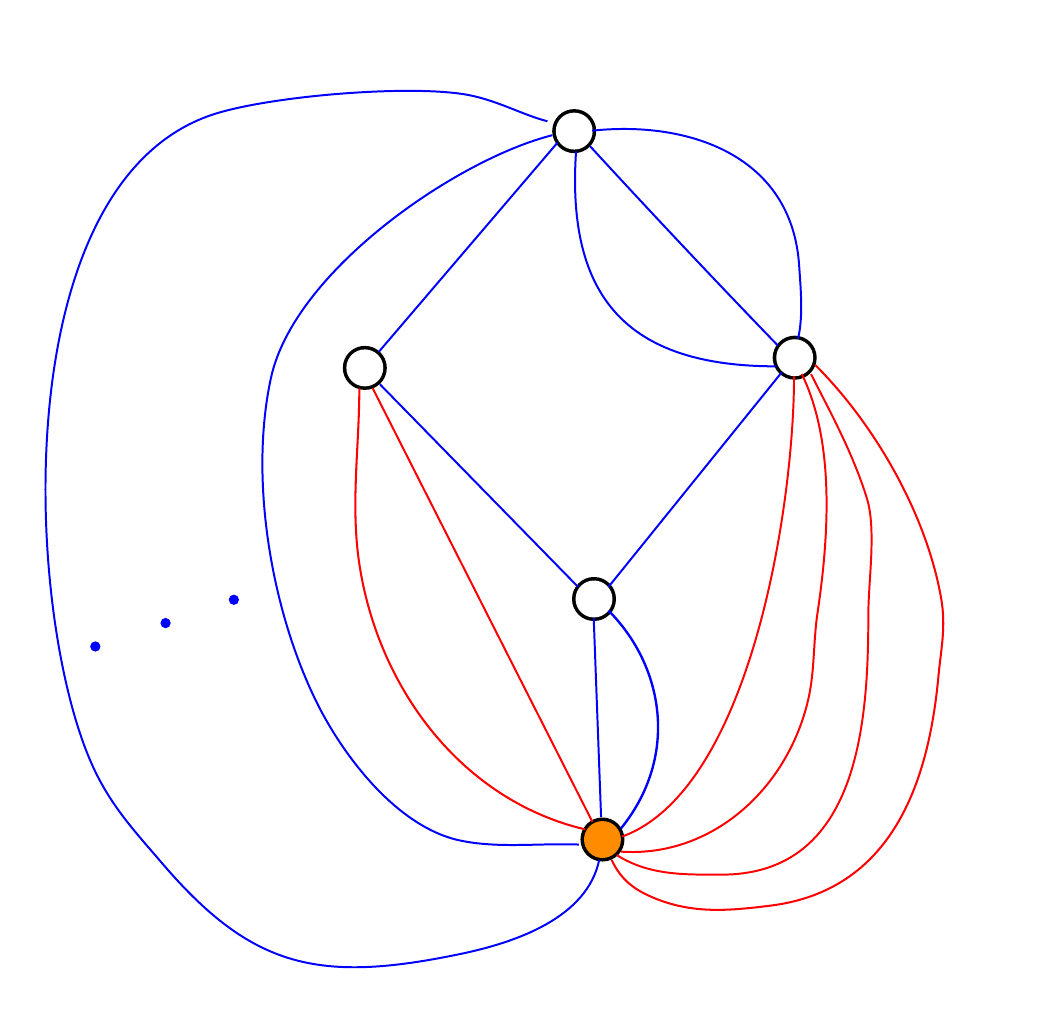}}%
    \put(0.53741905,0.88634751){\color[rgb]{0,0,0}\makebox(0,0)[lt]{\lineheight{1.25}\smash{\begin{tabular}[t]{l}$\color{black}v_1$\end{tabular}}}}%
    \put(0.3769978,0.62135233){\color[rgb]{0,0,0}\makebox(0,0)[lt]{\lineheight{1.25}\smash{\begin{tabular}[t]{l}$\color{black}v_2$\end{tabular}}}}%
    \put(0.78840333,0.63623198){\color[rgb]{0,0,0}\makebox(0,0)[lt]{\lineheight{1.25}\smash{\begin{tabular}[t]{l}$\color{black}v_3$\end{tabular}}}}%
    \put(0.54106077,0.44457247){\color[rgb]{0,0,0}\makebox(0,0)[lt]{\lineheight{1.25}\smash{\begin{tabular}[t]{l}$\color{black}v_4$\end{tabular}}}}%
    \put(0,0){\includegraphics[width=\unitlength,page=2]{PicirriloLiu.pdf}}%
  \end{picture}%
\endgroup%
}
\caption{Construction of $W_n$ for Example \ref{example1}. The edges with color blue (resp.~red) has sign $+$ (resp.~$-$). Applying the median constructions to $W_n$ gives us $K_n$ as seen in Figure \ref{Figure:Myexample}.}\label{fig:Wn}
\end{figure}
    
    Now we can use a method proposed by Greene \cite{Greene2008AST} to compute a presentation of $\pi_1(Y_{D_n})$ (also see Marengon\cite{Marengon} where the method is used for the same purpose). Applying this in our case gives us the following.
    $$\pi_1(Y_{D_n})=\langle x_1,\cdots,x_4 | r_1,\cdots,r_4 \rangle $$
    $$r_1=(x_1)^{2n+1}(x_2^{-1}x_1)(x_3^{-1}x_1)^{3}$$
    $$r_2=(x_1^{-1}x_2)(x_2)^{-2}(x_4^{-1}x_2)$$
    $$r_3=(x_1^{-1}x_3)^{3}(x_4^{-1}x_3)(x_3)^{-4}$$
    $$r_4=(x_3^{-1}x_4)(x_2^{-1}x_4)(x_4)^{2}$$
    We then can show that $\pi_1(Y_{D_n}) / \langle \langle x_3 \rangle \rangle$ is the trivial group.
    Using the relation $x_3=e$ to simplify $r_1$ and $r_3$, we will have $x_2=x_1^{2n+5}$ and $x_4=x_1^{-3}$. Now using these two new relation, we can turn $r_2$ and $r_4$ to relations only containing $x_1$. The result will be the following two relations: 
    $$x_1^{2}=e,$$
    $$x_1^{-2n-17}=e.$$
    Since $\text{gcd}(2,-2n-17)=1$, we will have that $x_1=e$ which in turn means $x_2=x_4=e$. This gives us the desired result. 
\end{proof}

\bibliography{bibtemplate}

\providecommand{\bysame}{\leavevmode ---\ }
\providecommand{\og}{``}
\providecommand{\fg}{''}
\providecommand{\smfandname}{et}
\providecommand{\smfedsname}{\'eds.}
\providecommand{\smfedname}{\'ed.}
\providecommand{\smfmastersthesisname}{M\'emoire}
\providecommand{\smfphdthesisname}{Th\`ese}
\begin{thebibliography}{HKL16}

\bibitem[Auc97]{auckly1997surgery}
{\scshape D.~Auckly} -- {\og Surgery numbers of 3-manifolds: a hyperbolic example\fg}, \emph{Geometric topology (Athens, GA, 1993)} \textbf{2} (1997), p.~21--34.

\bibitem[Aza23]{azarpendar2023negative}
{\scshape S.~Azarpendar} -- {\og Negative definite spin filling and branched double covers\fg}, 2023.

\bibitem[Gre08]{Greene2008AST}
{\scshape J.~E. Greene} -- {\og A spanning tree model for the heegaard floer homology of a branched double‐cover\fg}, \emph{Journal of Topology} \textbf{6} (2008).

\bibitem[HKL16]{hom2016surgery}
{\scshape J.~Hom, {\c{C}}.~Karakurt {\normalfont \smfandname} T.~Lidman} -- {\og Surgery obstructions and heegaard floer homology\fg}, \emph{Geometry \& Topology} \textbf{20} (2016), no.~4, p.~2219--2251.

\bibitem[HL18]{hom2018note}
{\scshape J.~Hom {\normalfont \smfandname} T.~Lidman} -- {\og A note on surgery obstructions and hyperbolic integer homology spheres\fg}, \emph{Proceedings of the American Mathematical Society} \textbf{146} (2018), no.~3, p.~1363--1365.

\bibitem[LB90]{LinesBoyer+1990+181+220}
{\scshape D.~Lines {\normalfont \smfandname} S.~Boyer} -- {\og Surgery formulae for casson's invariant and extensions to homology lens spaces.\fg}, \emph{Journal für die reine und angewandte Mathematik (Crelles Journal)} \textbf{1990} (1990), no.~405, p.~181--220.

\bibitem[LP24]{liu2024bounding}
{\scshape B.~Liu {\normalfont \smfandname} L.~Piccirillo} -- {\og Bounding the dehn surgery number by 10/8\fg}, 2024.

\bibitem[Mar16]{Marengon}
{\scshape M.~Marengon} -- {\og On d-invariants and generalized kanenobu knots\fg}, \emph{Journal of Knot Theory and Its Ramifications} \textbf{25} (2016), no.~08, p.~1650048.

\bibitem[McC14]{McCoy2014NonintegerSA}
{\scshape D.~McCoy} -- {\og Non‐integer surgery and branched double covers of alternating knots\fg}, \emph{Journal of the London Mathematical Society} \textbf{92} (2014).

\bibitem[OS05]{SObranch}
{\scshape P.~Ozsváth {\normalfont \smfandname} Z.~Szabó} -- {\og On the {H}eegaard {F}loer homology of branched double-covers\fg}, \emph{Advances in Mathematics} \textbf{194} (2005), no.~1, p.~1--33.

\bibitem[Owe22]{owens2022knots}
{\scshape B.~Owens} -- {\og Knots and 4-manifolds\fg}, 2022.

\bibitem[Pol14]{Polyak}
{\scshape M.~Polyak} -- {\og {From 3-manifolds to planar graphs and cycle-rooted trees}\fg}, 2014.

\bibitem[SZ22]{sivek2022surgery}
{\scshape S.~Sivek {\normalfont \smfandname} R.~Zentner} -- {\og Surgery obstructions and character varieties\fg}, \emph{Transactions of the American Mathematical Society} \textbf{375} (2022), no.~5, p.~3351--3380.

\end{thebibliography}
\bibliographystyle{smfalpha}
\end{document}